\documentclass[12pt]{amsart}

\usepackage{amsmath,amssymb,url}
\usepackage{hyperref}
\usepackage{amsthm}

\newtheorem{Thm}{Theorem}[section] 
 
\newtheorem{Lem}[Thm]{Lemma} 
\newtheorem{Cor}[Thm]{Corollary} 
\theoremstyle{definition}
\newtheorem{Rem}[Thm]{Remark} 
 
\theoremstyle{definition}
\newtheorem{Def}[Thm]{Definition}
\numberwithin{equation}{section}

\newcommand{\Clg}{\mathrm{Clg}}
\newcommand{\Cid}{\mathrm{Cig}}
\newcommand{\tD}{\mathrm{tD}}

\newcommand{\Clo}{\mathrm{Clo}}

\newcommand{\rem}{\mathrm{rem}}

\newcommand{\FF}{{\mathbb F}}

\newcommand{\KK}{{\mathbb K}}
\newcommand{\ZZ}{{\mathbb{Z}}}
\newcommand{\NN}{{\mathbb{N}}}

\DeclareMathAlphabet\mathbfsl {T1}{cmr}{bx}{it}

\title[Clones on $\ZZ_{n}$]{Expansions of abelian squarefree groups}
\author{Stefano Fioravanti}

\address{Stefano Fioravanti,
	Institut f\"ur Algebra,
	Johannes Kepler Universit\"at Linz,
	4040 Linz,
	Austria}
\email{\tt stefano.fioravanti66@gmail.com}
\urladdr{http://www.jku.at/algebra}
\thanks{Supported by the Austrian Science Fund (FWF): P29931.}
\date{\today}

\begin{document}
	
	\begin{abstract}
		
		We investigate finitary functions from $\mathbb{Z}_{n}$ to $\mathbb{Z}_{n}$ for a squarefree number $n$. We show that the lattice of all clones on the squarefree set $\mathbb{Z}_{p_1\cdots p_m}$ which contain the addition of $\mathbb{Z}_{p_1\cdots p_m}$ is finite. We provide an upper bound for the cardinality of this lattice through an injective function to the direct product of the lattices of all $(\ZZ_{p_i}, \FF_i)$-linearly closed clonoids, $\mathcal{L}(\ZZ_{p_i}, \FF_i)$, to the $p_i+1$ power, where $\FF_i = \prod_{j \in \{1,\dots,m\}\backslash \{i\}}\ZZ_{p_j}$. These lattices are studied in \cite{Fio.CSOF2} and there we can find an upper bound for their cardinality. Furthermore, we prove that these clones can be generated by a set of functions of arity at most $\max(p_1,\dots,p_m)$.
		
	\end{abstract}
	\maketitle
	
	\section{Introduction}
	
	The investigation of the lattice of all clones on a set $A$ has been a fecund field of research in general algebra with results such as Emil Post's characterization of the lattice of all clones on a two-element set \cite{Pos.TTVI}. This branch was developed further, e. g., in \cite{Ros.MCOA,PK.FUR,Sze.CIUA} and starting from \cite{BJK.TCOC}, clones are used to study the complexity of certain constraint satisfaction problems (CSPs).
	
	The aim of this paper is to describe the lattice of those clones on the set $\ZZ_{n}$ that contain the operation of addition of $\ZZ_{n}$,  with $n$ squarefree. Thus we want to study the part of the lattice of all clones on $\ZZ_{n}$ which is above the clone of all linear mappings. 
	
	In \cite{Idz.CCMO} P. Idziak characterized the number of polynomial Mal'cev clones (clones containing the constants and a Mal'cev term) on a finite set $A$, which is finite if and only if $|A|\leq 3$. In \cite{Bul.PCCT} A. Bulatov shows a full characterization of all infinitely many polynomial clones on the sets $\ZZ_p \times \ZZ_p$ and $\ZZ_{p^2}$ that contain $+$, where $p$ is a prime. Moreover, a description of polynomial clones on $\ZZ_{pq}$ containing the addition for distinct primes $p$ and $q$ is given in \cite{AM.PCOG} and polynomial clones containing $+$ on $\ZZ_n$, for $n$ squarefree, are described in \cite{May.PCOS}.
	
	In \cite{Kre.CFSO} S. Kreinecker proved that there are infinitely many non-finitely generated clones above the clone $\Clo( \mathbb{Z}_p \times \mathbb{Z}_p , +)$ of term operations of the group $(\mathbb{Z}_p \times \mathbb{Z}_p , +)$  for any prime $p > 2$. 
	
	Let $C$ be a set of functions. We denote by $C^{[n]}$ the subset of $n$-ary functions in $C$. In this paper we will make often use of the concept of $(\FF,\KK)$-linearly closed clonoid as defined in \cite[Definition $1.1$]{Fio.CSOF2} (generalization of \cite[Definition $1.1$]{Fio.CSOF}). We recall this definition.
	
	\begin{Def}
		\theoremstyle{definition}
		\label{DefClo-2}
		Let $m,s \in \NN$, let $q_1,\dots,q_m,p_1,\dots p_s$ be powers of primes, and let $\KK= \prod_{i=1}^m\mathbb{F}_{q_i}$, $\FF= \prod_{i=1}^s\mathbb{F}_{p_i}$ be products of fields of orders $q_1,\dots,q_m,p_1,\dots p_s$. An \emph{$(\FF,\KK)$-linearly closed clonoid} is a non-empty subset $C$ of $\bigcup_{k \in \mathbb{N}} \prod_{i=1}^s\mathbb{F}_{p_i}^{{\prod_{j=1}^m\mathbb{F}_{q_j}^k}}$ with the following properties:
		
		\begin{enumerate}
			\item[(1)] for all $n \in \NN$, $\mathbfsl{a}, \mathbfsl{b} \in \prod_{i=1}^s\mathbb{F}_{p_i}$, and $f,g \in C^{[n]}$:
			
			\begin{equation*}
				\mathbfsl{a}f + \mathbfsl{b}g \in C^{[n]};
			\end{equation*}
			
			\item[(2)] for all $l,n \in \NN$, $f \in C^{[n]}$, $(\mathbfsl{x}_1,\dots,\mathbfsl{x}_m) \in \prod_{j=1}^m\mathbb{F}_{q_j}^l$, and $A_i \in \mathbb{F}^{n \times l}_{q_i}$:
			
			\begin{equation*}
				g\colon (\mathbfsl{x}_1,\dots,\mathbfsl{x}_m) \mapsto f(A_1\cdot \mathbfsl{x}_1^t,\cdots,A_m\cdot \mathbfsl{x}_m^t) \text{ is in } C^{[l]},
			\end{equation*}
			
		\end{enumerate}
		where with the juxtaposition $\mathbfsl{a}f$ we denote the Hadamard product of the two vectors (i.e. the component-wise product $(a_1,\dots,a_n)\cdot (b_1,\dots,b_n) = (a_1b_1,\dots,a_nb_n)$).
	\end{Def}
	
	In \cite[Theorems $1.2$ and $1.3$]{Fio.CSOF2} we can find a complete description of the lattice of all $(\FF,\KK)$-linearly closed clonoids with $\FF$ and $\KK$ products of finite fields of pair-wise coprime order.
	
	The main result of this chapter regards the cardinality of the lattice of all clones on $\ZZ_{s}$ that contain $\Clo(\ZZ_{s},+)$, where $s$ is squarefree.
	
	\begin{Thm}
		\label{Thmgeneralembedding-2}
		Let $s = p_1\cdots p_m$ be a product of distinct primes and let $\FF_i = \prod_{j \in [m]\backslash \{i\}}\ZZ_{p_j}$ for all $1 \leq i \leq n$. Then there is an injective function from the lattice $\mathcal{L}( \ZZ_{s},+)$ of all clones containing $\Clo( \ZZ_{s},+)$, to the direct product of the lattices of all $(\ZZ_{p_i}, \FF_i)$-linearly closed clonoids, $\mathcal{L}(\ZZ_{p_i}, \FF_i)$, to the $p_i+1$ power, i. e:
		
		\begin{equation*}
			\mathcal{L}( \ZZ_{s},+)\hookrightarrow \prod_{i=1}^n\mathcal{L}(\ZZ_{p_i},\FF_i)^{p_i+1}.		
		\end{equation*}
	\end{Thm}
	
We will prove Theorem \ref {Thmgeneralembedding-2} in Section \ref{SAGB}. Vice versa, we find also an embedding of the lattice of all $(\ZZ_{p_1},$ $\prod_{i =2}^m\ZZ_{p_i})$-linearly closed clonoids into the lattice of all clones above $\Clo(\ZZ_{p_1\cdots p_m},+)$, where $p_1,\dots,p_m$ are not necessarily distinct prime numbers.

\begin{Thm}
	\label{ThEmbClonoids-2}
	Let $p_1,\dots,p_m$ be prime numbers and let $\FF_1 = \prod_{i =2}^m\ZZ_{p_i}$. Then the lattice of all $(\ZZ_{p_1},\FF_1)$-linearly closed clonoids is embedded in the lattice of all clones above $\Clo(\ZZ_{p_1\cdots p_m},+)$.
\end{Thm}

	From these results we can obtain bounds for the number of clones on $\ZZ_{s}$ that contain $\Clo(\ZZ_{s},+)$.
	
	\begin{Cor}
		\label{Corfinale-2}
		Let $s = p_1\cdots p_m$ be a product of distinct primes and let $\FF_i = \prod_{j \in [n]\backslash \{i\}} \ZZ_{p_j}$. Then the number of clones containing $\Clo(\ZZ_{s},+)$ is bounded by:
		
		\begin{equation*}
			\sum_{i =1}^m|\mathcal{L}(\ZZ_{p_i}, \FF_i)|  -m +1\leq |\mathcal{L}(\ZZ_{s},+)| \leq \prod_{i=1}^m|\mathcal{L}(\ZZ_{p_i}, \FF_i)|^{p_i+1}.
		\end{equation*}
	\end{Cor}

	We will prove Corollary \ref {Corfinale-2} in Section \ref{SAGB}. This corollary extends the finiteness results of \cite{May.PCOS} for clones containing $\Clo( \ZZ_s,+ )$ which do not necessarily contain constants with $s$ squarefree. 
	
	We can also use Theorem \ref{Thmgeneralembedding-2} to find a concrete bound on the arity of the generators of clones containing $\Clo(\ZZ_{s},+)$, with $s$ squarefree.
	
	\begin{Cor}
		\label{CorArFun-2}
		Let $s=p_1\cdots p_m$ be a product of distinct prime numbers. Then every clone containing $\Clo(\ZZ_{s},+)$ can be generated by a set of functions of arity at most $\max(p_1,\dots,p_m)$.
	\end{Cor}
	
	The last theorem states that there is a dichotomy for the cardinalities of the clones of finite expanded abelian groups.
	
	\begin{Thm}\label{ThmDichot}
		
		Let $\mathbf{G}$ be a finite abelian group. Then $\mathbf{G}$ has finitely many expansions up to term equivalence or, equivalently, the lattice of all clones containing $\Clo( G,+,-,0)$ is finite if and only if $\mathbf{G}$ is of squarefree order.
	\end{Thm}
	
	We will prove Theorem \ref {ThmDichot} in Section \ref{SAGB}. Thus this theorem shows a surprising dichotomy about the cardinalities of the clones of finite expanded abelian groups up term equivalence. Indeed, the order of a group seems to have no connection in principle with the finiteness of the lattice of all distinct clones up term equivalence above the linear mappings on it.

	\section{Preliminaries and notation}\label{Preliminaries3}
	
	We use boldface letters for vectors, e. g., \index{$\mathbfsl{u}$}$\mathbfsl{u} = (u_1,\dots,u_n)$ for some $n \in \NN$. Moreover, we will use $\langle\mathbfsl{v}, \mathbfsl{u}\rangle$ for the scalar product of the vectors $\mathbfsl{v}$ and $\mathbfsl{u}$. Let $A$ be a set and let $0_A \in A$. We denote by $\mathbf{0}_{n}$ a constant $0_A$ vector of length $n$.
	
	We denote by \index{$[n]$}$[n]$ the set $\{i \in \NN\mid 1 \leq i \leq n\}$ and by \index{$[n]_0$}$[n]_0$ the set $[n] \cup \{0\}$. Moreover we denote by \index{$\NN_0$}$\NN_0$ the set $\NN \cup \{0\}$. Let $\mathbfsl{x} \in \ZZ_p^n$ and let $\mathbfsl{a} \in [p-1]_0^n$. Then we denote by \index{$\mathbfsl{x}^{\mathbfsl{a}}$}$\mathbfsl{x}^{\mathbfsl{a}}$ the product $\prod_{i =1}^nx_i^{a_i}$. We use also convention that an empty product is $1$.
	
	From now on we will consider the group $\prod_{i=1}^m\ZZ_{p_i}$ instead of $\mathbb{Z}_{s}$, where $s = \prod_{i=1}^mp_i$ is squarefree. We can observe that the two groups are isomorphic and thus equivalent for our purpose. 
	
	Moreover, we consider $\prod_{i=1}^m\ZZ_{p_i}^n$ instead of $(\prod_{i=1}^m\ZZ_{p_i})^n$ as the domain of the $n$-ary functions we want to study.
	
	Let $S$ be a set of finitary functions from a group $G$ to itself . We denote by $\Clg(S)$ the \emph{clone generated by} $S \cup \{+\}$ on $G$. Let $\KK$ and $\FF$ be product of finite fields with pair-wise coprime order. We write \index{$\Cid(F)$}$\Cid(F)$ for the $(\FF,\KK)$-linearly closed clonoid generated by a set of functions $F \subseteq \bigcup_{k\in\NN} \FF^{\KK^k}$, as defined in \cite{Fio.CSOF2}. 
	
	\section{Facts about clones}
	
	In this paper we want to study sets of finitary functions from $\prod_{i=1}^m\ZZ_{p_i}$ to itself. The sets of functions that we want to study are the clones containing $\Clo(\ZZ_{s},+)$, where $s$ is a product of distinct primes. 
	
	Furthermore, let $n \in \NN$. We denote by  $\mathcal{L}(\ZZ_{n},+)$ the lattice of all clones containing $\Clo( \ZZ_{n},+)$.
	
	In \cite{May.PCOS} we can find a description for polynomial clones (clones containing all constants) which contain $\Clo(\ZZ_{s},+)$, where $s$ is a product of distinct primes. With a different strategy we will show a characterization that extends the finiteness result in \cite{May.PCOS} to those clones of finite abelian groups that do not necessarily contain all constants.
	
	Let us now show some basic facts about finitary functions from $\ZZ_{n}$ to $\ZZ_{n}$. 
	
	\begin{Rem}\label{RemPolCom}
		\theoremstyle{definition}
		It is a well-known fact that every finite field is polynomially complete. Thus for all $f\colon \FF_p^n \rightarrow \FF_p$, there exists a sequence $\{a_{\mathbfsl{m}}\}_{\mathbfsl{m} \in [p-1]_0^n} \subseteq \FF_p^n$ such that for all $\mathbfsl{x} \in \FF_p^n$, $f$ satisfies:
		\begin{equation*}
			f(\mathbfsl{x}) = \sum_{\mathbfsl{m} \in [p-1]_0^n}a_{\mathbfsl{m}}\mathbfsl{x}^{\mathbfsl{m}}.
		\end{equation*}
	\end{Rem}
	
	We can observe that if $p_1,\dots,p_m$ are distinct prime numbers we can split a function $f\colon \prod_{i=1}^m\ZZ_{p_i}^n \rightarrow \prod_{i=1}^m\ZZ_{p_i}$ in $f = \sum_{i=1}^m f_i$, where $f_i = \prod_{j\in [m]\backslash \{i\}}p_j^{p_i-1}f$. This implies, for example, that we can prove the following remark.
	
	\begin{Rem}\label{RemLinComb}
		
		Let $p_1 \cdots p_m =s$ be a product of distinct prime numbers and let $C$ be a clone containing $\Clo( \ZZ_{s}, +)$. Then for all $k \in \NN$ and $(\mathbfsl{a}_1,\dots,\mathbfsl{a}_m) \in \prod_{i=1}^m\ZZ_{p_i}^k$, $h_{(\mathbfsl{a}_1,\dots,\mathbfsl{a}_m)}\colon  \prod_{i=1}^m\ZZ_{p_i}^k \rightarrow  \prod_{i=1}^m\ZZ_{p_i}$ defined by:
		
		\begin{equation*}
			h_{(\mathbfsl{a}_1,\dots,\mathbfsl{a}_m)}\colon (\mathbfsl{x}_1,\dots,\mathbfsl{x}_m) \mapsto (\langle\mathbfsl{a}_1,\mathbfsl{x}_1\rangle,\dots, \langle \mathbfsl{a}_m,\mathbfsl{x}_m\rangle)
		\end{equation*}
		is in $C$.
		
	\end{Rem}
	
	Let $A$ be a set and let $\FF_p$ be a field of order $p$. With the following lemma we show that every function from $\FF_p^n \times A^s$ to $\FF_p$ can be seen as the induced function of a polynomial of $\mathbf{R}[x_1,\dots,x_n]$, where $\mathbf{R} = \FF_p^{A^s}$. This easy fact will be often used later. 
	
	\begin{Lem}\label{Lem2Genexprofa}
		Let $A$ be a set and let $\FF_p$ be a field of order $p$. Then for every function $f$ from $\FF_p^n \times A^s$ to $\FF_p$ there exists a sequence of functions $\{f_{\mathbfsl{m}}\}_{\mathbfsl{m} \in [p-1]_0^n}$ from $A^s$ to $\FF_p$ such that $f$ satisfies for all $\mathbfsl{x} \in \FF_p^n$, $\mathbfsl{y} \in A^s$:
		
		\begin{equation*}
			f(\mathbfsl{x}, \mathbfsl{y}) = \sum_{\mathbfsl{m} \in [p-1]_0^n} f_{\mathbfsl{m}}(\mathbfsl{y})\mathbfsl{x}^{\mathbfsl{m}}.
		\end{equation*}

	\end{Lem}
	
	The previous lemma in our setting implies the following.
	
	\begin{Lem}\label{Lem2Genexprofb}
		Let $p_1,\dots,p_m$ be distinct prime numbers. Then for every function $f$ from $\prod_{i=1}^m\ZZ_{p_{i}}^n$ to $\prod_{i=1}^m\ZZ_{p_i}$ there exist $m$ sequences of functions $\{f_{(i,{\mathbfsl{h}_i})}\}_{\mathbfsl{h}_i \in [p_i-1]_0^n}$ from $\prod_{j \in [m]\backslash \{i\}}\ZZ_{p_j}^n$ to $\ZZ_{p_i}$, for all $i \in [m]$, such that $f$ satisfies for all $(\mathbfsl{x}_1,\dots,\mathbfsl{x}_m) \in \prod_{i=1}^m\ZZ_{p_{i}}^n$:
		
		\begin{align*}
			\label{GeneralExproffb}
			f(\mathbfsl{x}_1,\dots,\mathbfsl{x}_m) &= (\sum_{\mathbfsl{h}_1 \in [p_1-1]_0^n} f_{(1,{\mathbfsl{h}_1})}(\mathbfsl{x}_2,\dots,\mathbfsl{x}_m)\mathbfsl{x}_1^{\mathbfsl{h}_1}, \dots,
			\\ &\sum_{\mathbfsl{h}_m \in [p_m-1]_0^n} f_{(m,{\mathbfsl{h}_m})}(\mathbfsl{x}_1,\dots,\mathbfsl{x}_{m-1})\mathbfsl{x}_m^{\mathbfsl{h}_m}).
		\end{align*} 
	\end{Lem}

	\section{Embedding of the Clonoids}

	The aim of this section is to prove that for all $i \in [m]$ there exists an embedding of the lattice of all $(\mathbb{Z}_{p_i},\FF_i)$-linearly closed clonoids in the lattice of all clones containing $\Clo( \prod_{i \in [m]}\ZZ_{p_i},+)$, where $p_1,\dots,p_m$ are prime numbers and $\FF_i = \prod_{j \in [m]\backslash \{i\}} \ZZ_{p_j}$. This clearly provides a lower bound for the cardinality of the lattice of all clones containing $\Clo( \ZZ_{n},+)$ when $n$ is squarefree.
	
	For all $f \in \ZZ_{p_1}^{\FF_1^n}$ we define $e(f): \prod_{j=1}^m\ZZ_{p_j}^n \rightarrow \prod_{j=1}^m\ZZ_{p_j}$ by: 
	
	\begin{align*}
		&e(f) \colon  (\mathbfsl{x}_1,\dots,\mathbfsl{x}_m) \mapsto
 (f(\mathbfsl{x}_2,\dots,\mathbfsl{x}_m),0_{\ZZ_{p_{2}}},\dots,0_{\ZZ_{p_m}})
	\end{align*}
	for all $(\mathbfsl{x}_1, \dots,\mathbfsl{x}_m) \in \prod_{j=1}^m\ZZ_{p_j}^n$.
	
	Furthermore, we define $\gamma$ from the lattice of all $(\mathbb{Z}_{p_1},\FF_1)$-linearly closed clonoids to the lattice of all clones containing $\Clo( \prod_{i \in [m] }\ZZ_{p_i},+)$ such that for all $C \in \mathcal{L}(\mathbb{Z}_{p_1},\FF_1)$:
	
	\begin{equation}\label{equ:1-2}	
		\begin{split}
			\gamma(C) := \bigcup_{n \in \NN}\{e(g)+ h_{(\mathbfsl{a}_1,\dots,\mathbfsl{a}_m)} \mid g \in C^{[n]}, (\mathbfsl{a}_1,\dots,\mathbfsl{a}_m) \in \prod_{j=1}^m\ZZ_{p_j}^n\}
		\end{split}
	\end{equation}
	where $h_{(\mathbfsl{a}_1,\dots,\mathbfsl{a}_m)}$ is defined in Remark \ref{RemLinComb}.
	
	In order to prove Theorem \ref{ThEmbClonoids-2} we first present an easy lemma omitting the proof.
	
	\begin{Lem}
		\label{Lemclogen-2}
		Let $\FF = \prod_{i=1}^s \FF_{p_i}$ and $\KK= \prod_{i=1}^m \FF_{q_i}$ be products of finite fields. Let $X \subseteq \bigcup_{n \in \NN} \mathbb{F}^{\mathbb{K}^n}$. Then $\Cid(X) = \bigcup_{n \in \NN} X_n$ where:
		\begin{flushleft}
			$X_0 := X$
			\\$X_{n+1} := \{\mathbfsl{a}f + \mathbfsl{b}g \mid \mathbfsl{a},\mathbfsl{b} \in \FF, f,g \in X_n^{[r]}, r \in \NN\} \cup \{g: (\mathbfsl{y}_1,\dots,\mathbfsl{y}_m) \mapsto f(A_1\cdot \mathbfsl{y}_1^t,\cdots,A_m\cdot \mathbfsl{y}_m^t) \mid f \in X_n^{[k]}, A_i \in \mathbb{F}^{k \times l}_{q_i}\}$.
		\end{flushleft}
	\end{Lem}
	
	We omit the straightforward proof of this Lemma which allows us to prove Theorem \ref{ThEmbClonoids-2}.
	
	\begin{proof}[Proof of Theorem \ref{ThEmbClonoids-2}]

		Let $\gamma$ be the function defined in \eqref{equ:1-2}. First we show that $\gamma$ is well-defined and then we show that $\gamma$ is injective and that $\gamma$ is a homomorphism. Let $C$ be a $(\ZZ_{p_1},\FF_1)$-linearly closed clonoid. Clearly $\gamma(C)$ contains the projections and the binary addition on $\prod_{i=1}^m\ZZ_{p_i}$. Moreover, let $f,f_1,\dots,f_n \in \gamma(C)$ be an $n$-ary and $n$ $s$-ary functions respectively. Then there exist $g_f,g_1,\dots,g_n \in C$, $(\mathbfsl{a}_1,\dots,\mathbfsl{a}_m)\in \prod_{i=1}^m\ZZ_{p_i}^n$, and $(\mathbfsl{a}_{(1,j)},\dots,\mathbfsl{a}_{(m,j)}) \in \prod_{i=1}^m\ZZ_{p_i}^s$, for all $j \in [n]$, such that:
		
		\begin{equation*}
			f(\mathbfsl{x}_1,\dots,\mathbfsl{x}_m) = (\langle \mathbfsl{a}_1,\mathbfsl{x}_1 \rangle + g_f(\mathbfsl{x}_2,\dots,\mathbfsl{x}_m), \langle \mathbfsl{a}_2,\mathbfsl{x}_2 \rangle, \dots,\langle \mathbfsl{a}_m,\mathbfsl{x}_m \rangle ),
		\end{equation*}	
		for all $(\mathbfsl{x}_1,\dots,\mathbfsl{x}_m)\in \prod_{i=1}^m\ZZ_{p_i}^n$ and for all $1 \leq j \leq n$:
		
		\begin{equation*}
			f_j(\mathbfsl{y}_1,\dots,\mathbfsl{y}_m) = (\langle \mathbfsl{a}_{(1,j)},\mathbfsl{y}_1 \rangle + g_j(\mathbfsl{y}_2,\dots,\mathbfsl{y}_m), \langle \mathbfsl{a}_{(2,j)},\mathbfsl{y}_2 \rangle, \dots,\langle \mathbfsl{a}_{(m,j)},\mathbfsl{y}_m \rangle)
		\end{equation*}
		for all $(\mathbfsl{y}_1,\dots,\mathbfsl{y}_m)\in \prod_{i=1}^m\ZZ_{p_i}^s$. Then $h = f \circ (f_1,\dots,f_n)$ can be written as:
		
		\begin{equation*}
			h(\mathbfsl{y}_1,\dots,\mathbfsl{y}_m) = (\langle \mathbfsl{c}_1,\mathbfsl{y}_1 \rangle + g_h(\mathbfsl{y}_2,\dots,\mathbfsl{y}_m), \langle \mathbfsl{c}_2,\mathbfsl{y}_2 \rangle, \dots,\langle \mathbfsl{c}_m,\mathbfsl{y}_m \rangle ),
		\end{equation*}	
		where for all $u\in [m]$, $j \in [s]$, $(\mathbfsl{c}_u)_j = \sum_{i =1}^n(\mathbfsl{a}_u)_i(\mathbfsl{a}_{(u,i)})_j$ and $g_h\colon \prod_{i =2}^m\ZZ_{p_i}^s$ $ \rightarrow \ZZ_{p_1}$ is defined by:
		
		\begin{equation*}
			\begin{split}
				g_h(\mathbfsl{y}_2,\dots,\mathbfsl{y}_m) =& \langle \mathbfsl{a}_1,\mathbf{d}(\mathbfsl{y}_2,\dots,\mathbfsl{y}_m) \rangle + g_f(\langle \mathbfsl{a}_{(2,1)},\mathbfsl{y}_2 \rangle, \dots, \langle \mathbfsl{a}_{(2,n)},\mathbfsl{y}_2 \rangle\\&, \dots,\langle \mathbfsl{a}_{(m,1)}, \mathbfsl{y}_m \rangle, \dots, \langle \mathbfsl{a}_{(m,n)},\mathbfsl{y}_m \rangle),
			\end{split}
		\end{equation*}
		with $\mathbf{d}(\mathbfsl{y}_2,\dots,\mathbfsl{y}_m) =  (g_1(\mathbfsl{y}_2,\dots,\mathbfsl{y}_m),\dots,g_n(\mathbfsl{y}_2,\dots,\mathbfsl{y}_m))$ for all $(\mathbfsl{y}_2,$ $\dots,\mathbfsl{y}_m) \in \prod_{i =2}^m\ZZ_{p_i}^s$. We can see from Definition \ref{DefClo-2} that $g_h \in C$. Thus $\gamma(C)$ is closed under composition and $\gamma$ is well-defined. 
		
		Next we prove that $\gamma$ is injective. Let $C$ and $D$ be two $(\ZZ_{p_1},\FF_1)$-linearly closed clonoids such that $\gamma(C) = \gamma(D)$ and let $g \in C$ be an $l$-ary function. Then let $s\colon \prod_{i=1}^m\ZZ_{p_i}^l \rightarrow \prod_{i=1}^m\ZZ_{p_i}$ be such that $e(g) =s$. Then $s$ is in $\gamma(C) = \gamma(D)$. By definition of $\gamma$, this implies that $e(g) = e(g') + h_{(\mathbfsl{a}_1,\dots,\mathbfsl{a}_m)}$ for some $g' \in D$ and $(\mathbfsl{a}_1,\dots,\mathbfsl{a}_m) \in \prod_{i=1}^m\ZZ_{p_i}^l $. The only possibility is that  $g = g' \in D$ and thus $C \subseteq D$. We can repeat this argument for the other inclusion and hence $\gamma$ is injective. Furthermore, we have that for all $C,D \in \mathcal{L}(\ZZ_{p_1},\FF_1)$, $\gamma(C \cap D) = \gamma(C) \cap \gamma(D)$. We can observe that $\gamma$ is monotone, thus $\gamma(C \vee D) \supseteq \gamma(C) \vee \gamma(D)$. For the other inclusion we prove by induction on $n$ that $\gamma(C) \vee \gamma(D) \supseteq e(X_n)$, where $C \vee D = \bigcup_{n \in \NN} X_n$ with:
		\begin{flushleft}
			$X_0 = C \cup D$
			\\$X_{n+1} = \{af + bg \mid a,b \in \ZZ_{p_1}, f,g \in X_n^{[r]}, r \in \NN\} \cup \{g\colon (\mathbfsl{y}_2,\dots,\mathbfsl{y}_m) \mapsto f(A_2\cdot \mathbfsl{y}_2^t,\cdots,A_m\cdot \mathbfsl{y}_m^t) \mid f \in X_n^{[k]}, A_i \in \mathbb{Z}^{k \times l}_{p_i}, k,l \in \NN \}$.
		\end{flushleft}
		Base step $n = 0$: $e(C \cup D) = e(C) \cup e(D) \subseteq \gamma(C) \vee \gamma(D)$.
		
		Induction step $n >0$: suppose that the claim holds for $n-1$. Then let $g \in e(X_{n})$. Thus there exists $u \in X_n$ such that $e(u)=g$ and either $u$ is a linear combination of functions in $X_{n-1}$ or there exist  $f \in X_n^{[k]}, A_i \in \mathbb{Z}^{k \times l}_{p_i}$ for all $i \in [m]\backslash\{1\}$, and $k,l \in \NN$ such that $u \colon  (\mathbfsl{y}_2,\dots,\mathbfsl{y}_m) \mapsto f(A_2\cdot \mathbfsl{y}_2^t,\cdots,A_m\cdot \mathbfsl{y}_m^t)$. In both cases we have $g \in \Clg(e(X_{n-1}) \cup \bigcup_{t\in \NN} \{h_{(\mathbfsl{a}_1,\dots,\mathbfsl{a}_m)} \mid (\mathbfsl{a}_1,\dots,\mathbfsl{a}_m) \in \prod_{i \in [m]} \ZZ_{p_i}^t\}) \subseteq \gamma(C) \vee \gamma(D)$ and this concludes the induction proof. By Lemma \ref{Lemclogen-2}, $\gamma(C) \vee \gamma(D) \supseteq e(C \vee D)$. 
		
		We can observe that $\gamma(C\vee D)$ is the clone generated by $e(C \vee D)$, $+$ and all the mappings $h_{(\mathbfsl{a}_1,\dots,\mathbfsl{a}_m)}$ defined in Remark \ref{RemLinComb}. Since $e(C \vee D) \subseteq \gamma(C) \vee \gamma(D)$, it follows that $\gamma(C\vee D) \subseteq \gamma(C) \vee \gamma(D)$ . Hence $\gamma$ is an embedding.
	\end{proof}
	
	\section{A general bound}
	\label{SAGB}
	
	In the current section our goal is to determine a bound for the cardinality of the lattice of all clones containing $\Clo( \ZZ_{s},+)$, where $s =p_1\cdots p_m$ is a product of distinct primes. Theorem \ref{Cor3-2} gives a complete list of generators for a clone containing $\Clo(\ZZ_{s},+)$ that explains the connection between clonoids and clones in this case. The generators of Theorem \ref{Cor3-2} are substantially formed by a product of a unary member of a generating set of a $(\ZZ_{p_i},\prod_{j\in [m]\backslash \{i\}}\ZZ_{p_j})$-linearly closed clonoid and a monomial generating a clone on $\ZZ_{p_i}$ for $i \in [m]$. This puts together the characterization in \cite{Kre.CFSO}and \cite[Theorem $1.2$]{Fio.CSOF2} which are the main ingredients of this section.
	
	We start showing some lemmata which we need to prove that clones containing $\Clo( \ZZ_{s},+)$ are strictly characterized by the $(\ZZ_{p_i},\FF_i)$-linearly closed clonoids, where we denote by $\FF_i$ the product $\prod_{j \in [m]\backslash \{i\}}\ZZ_{p_j}$.
	
	In this section we have to deal with polynomials whose coefficients are finitary functions from $\FF_i$ to $\ZZ_{p_i}$. The next step will be to generalize some results in \cite{Kre.CFSO} about $p$-linearly closed clonoids to polynomials in a polynomial ring over a set of countably many variables. Let us start with the notation. Let $\mathbf{R}$ be a ring. We fix an alphabet $X := \{x_i \mid i \in \NN\}$ and we denote by $\mathbf{R}[X]$ the polynomial ring over $\mathbf{R}$ in the variables $X$.
	
	Following \cite{Kre.CFSO} we denote by $\tD(h)$ the \emph{total degree of a monomial} $h$, which is defined as the sum of the exponents. We also denote by $\tD(f) := \max(\{d \mid d = \tD(h),  h \text { is a monomial in }f\})$ the maximum of the total degrees of monomials in $f$. Let $f \in \mathbf{R}[x_1,\dots,x_k]$  and let $\mathbfsl{x} = (x_1,\dots,x_k)$. Then $f$ can be written as:
	
	\begin{equation*}
		f =\sum_{\mathbfsl{m} \in \NN_0^k}r_{\mathbfsl{m}}\mathbfsl{x}^{\mathbfsl{m}},
	\end{equation*}
	for some sequence $\{r_{\mathbfsl{m}}\}_{\mathbfsl{m} \in \NN_0^k}$ in $\mathbf{R}$ with only finitely many non-zero members and where $\mathbfsl{x}^{\mathbfsl{m}}=\prod_{i=1}^nx_i^{m_i}$.
	
	Next we introduce a notation for the composition of multivariate polynomials. Let $l,h \in \NN$, $g,f_1,\dots,f_h \in \mathbf{R}[x_1,\dots,x_l]$, and let $\mathbfsl{b} = (b_1, \dots, b_h ) \in \NN^h$ with $1 \leq b_1 < b_2 < \cdots < b_h \leq l$. Then we define $g \circ_{\mathbfsl{b}} (f_1 ,\dots, f_h)$ by:
	
	\begin{equation*}
		g \circ_{\mathbfsl{b}} (f_1,\dots,f_h) := g(x_1,\dots, x_{b_1-1}, f_1, x_{b_1+1},\dots,x_{b_2-1}, f_2, x_{b_2+1},\dots ).
	\end{equation*}
	Let $\mathbf{R}[X]$ be a polynomial ring and let $f \in \mathbf{R}[X]$. Since later we want to introduce the induced function of a polynomial, in order to have a unique polynomial for every induced function, we consider the ideal $I$ generated by the polynomials $x_i^p-x_i$ in $\mathbf{R}[X]$, for every $x_i \in X$. By \cite[Chapter $15.3$]{Eis.CA} there is a unique remainder $\rem(f)$ of $f$ with respect to $I$. This remainder has the property that the exponents of the variables are less or equal $p-1$. Following \cite[Section $2$]{Kre.CFSO}, we define
	
	\begin{equation*}
		\mathbf{R}[X]_p := \{ \sum_{\mathbfsl{m} \in [p-1]_0^k}r_{\mathbfsl{m}}\mathbfsl{x}^{\mathbfsl{m}}\mid k \in \NN_0, r_{\mathbfsl{m}}\in R, \mathbfsl{x} = (x_1,\dots,x_k)\}
	\end{equation*} 
	We can observe that these polynomials form a set of representatives of the set of all classes of the quotient $\mathbf{R}/I$.
	
	With the next definition we want to introduce sets of polynomials in polynomial rings closed under composition from the right and from the left with linear mappings. 
	
	\begin{Def}
		\label{DefPolyClonoid}
		Let $A$ be a set and let $\mathbf{R}$ be a ring. Let $\mathbf{R}^A[X]$ be a polynomial ring. An $\mathbf{R}^A$-\emph{polynomial linearly closed clonoid} is a non-empty subset $C$ of $\mathbf{R}^A[X]$ with the following properties:
		
		\begin{enumerate}
			\item[(1)] for all $f \in C$, $g \in C$, and $a,b \in  \mathbf{R}$
			\begin{equation*}
				af+ bg \in C;
			\end{equation*}
			
			\item[(2)] for all $s \in \NN$, $f \in C \cap \mathbf{R}^A[x_1,\dots,x_s]$, and $M \in \mathbf{R}^{ s\times l}$:
			
			\begin{equation*}
				g = f(M \cdot (x_1,\dots,x_l)^t) \text{ is in } C.
			\end{equation*}		
		\end{enumerate}
		
	\end{Def}
	
	We can observe that item $(2)$ of Definition \ref{DefPolyClonoid} implies that for all $s,k \in \NN$, $l \leq s$, $f \in C \cap \mathbf{R}^A[x_1,\dots,x_s]$, and $\mathbfsl{a} \in \mathbf{R}^{ k}$:
	
	\begin{equation*}
		g = f \circ_{(l)}(\prod_{i\in [k]} a_ix_i) \text{ is in } C.
	\end{equation*}		
	Let $S \subseteq \mathbf{R}^A[X]$. Then we denote by \index{$\langle \rangle_{\mathbf{R}^A}$}$\langle S\rangle_{\mathbf{R}^A}$ the $\mathbf{R}^A$-\emph{polynomial linearly closed clonoid} generated by $S$. We can see that $\mathbf{R}^A[X]_p$ forms an $\mathbf{R}^A$-linearly closed clonoid.
	
	Let us now modify \cite[Lemmata $3.8$ and $3.9$]{Kre.CFSO} to deal with $\mathbf{R}^A$-polynomial linearly closed clonoids. Indeed \cite[Lemmata $3.8$]{Kre.CFSO} is stated for $\ZZ_p$-polynomial linearly closed clonoids and works for $\mathbf{R}^A$-polynomial linearly closed clonoids in general with essentially the same proof as in \cite{Kre.CFSO}.
	
	\begin{Lem}
		\label{Lempclonoids}
		Let $A$ be a set and let $\mathbf{R}$ be a ring. Let $d \in \NN$, let $r \in \mathbf{R}^A$ and let $g \in\mathbf{R}^A[X]$ with $\tD(g)$ $\leq d$, and the coefficient of $\mathbfsl{x}^{\mathbf{1}_d}$ in $g$ is $0$. Then $rx_1\cdots x_d\in \langle\{rx_1\cdots x_d + g\}\rangle_{\mathbf{R}^A}$.
	\end{Lem}
	
	\begin{proof}
		Let $g \in \mathbf{R}^A[X]$ and let $C := \langle\{rx_1\cdots x_d + g\}\rangle_{\mathbf{R}^A}$. By setting all variables $x_i$ with $i > d$ to $0$, we may assume that $g \in \mathbf{R}^A[x_1 , \dots, x_d]$.
		
		Next we proceed by induction on the number of monomials of $g$ in order to show that $rx_1\cdots x_d \in \langle \{rx_1\cdots x_d + g\}\rangle_{\mathbf{R}^A} \subseteq C$. 
		
		If $g=0$ then the claim obviously holds. Let us suppose that $rx_1\cdots x_d\in \langle\{rx_1\cdots x_d + s\}\rangle_{\mathbf{R}^A}$ for every $s$ with $t \geq 0$ monomials. Let the number of monomials of $g$ be $t+1$. We observe that there exist $x_l \in \{x_1,\dots,x_d\}$ and a monomial $m$ of $g$ such that $x_l$ does not appear in $m$. Thus we obtain:
		\begin{equation*}
			rx_1\cdots x_d + g - (rx_1\cdots x_d + g) \circ_{(l)} 0 = rx_1\cdots x_d + g - g \circ_{(l)} 0\in C.
		\end{equation*}
		Thus $g' := g - g \circ_{(l)} 0$ satisfies the properties that $\tD(g') \leq d$, the coefficient of $\mathbfsl{x}^{\mathbf{1}_d}$ in $g'$ is $0$, $g'\in \mathbf{R}^A[x_1 ,\dots, x_d ]$ and $g'$ has fewer monomials than $g$, since the monomial $m$ is cancelled in $g - g \circ_{(l)} 0$. By the induction hypothesis $rx_1\cdots x_d\in \langle\{rx_1\cdots x_d + g'\}\rangle_{\mathbf{R}^A} \subseteq \langle\{rx_1\cdots x_d + g\}\rangle_{\mathbf{R}^A}$ and the claim holds.
	\end{proof}

	With Lemma \ref{Lempclonoids} we can now prove the following generalization of \cite[Lemma $3.9$]{Kre.CFSO} with the same proof. The following Lemma generalizes \cite[Lemma $3.9$]{Kre.CFSO}, which is stated for $\ZZ_p$-polynomial linearly closed clonoids, to $\ZZ_{p}^n$-polynomial linearly closed clonoids.
	
	\begin{Lem}
		\label{LemFoundPcloni}
		Let $d,n \in \NN$, let $p$ be a prime and let $f$ be a polynomial in $\ZZ_{p}^n[X]_p$ with $d := \tD(f)$. Let $m$ be a monomial with coefficient $r \in \ZZ_{p}^n$ and $\tD(m) = d$. Then:
		\begin{equation*}
			rx_1 \dots x_d \in \langle\{f\}\rangle_{\ZZ_{p}^n}.
		\end{equation*}
	\end{Lem}
	
	\begin{proof}
		Let $f = \sum_{\mathbfsl{m} \in [p-1]_0^l}r_{\mathbfsl{m}}\mathbfsl{x}^{\mathbfsl{m}} \in \ZZ_{p}^n[x_1,\dots,x_l]_p$, $C := \langle\{f\}\rangle_{\ZZ_{p}^n}$, let $d := \tD(f)$, and let $m = r\mathbfsl{x}^{\mathbfsl{s}}$ be a monomial of $f$ with $\tD(m) = d$. Without loss of generality we suppose $s_j > 0$ for all $j \in [u]$ for some $u \leq l$ and $s_j =0$ otherwise. We prove by case distinction that $rx_1 \dots x_d \in \langle\{f\}\rangle_{\ZZ_{p}^n}$.
		
		Case $s_j =1$ for all $j \in [u]$: then clearly there exists $g$ with $\tD(g) \leq d$ such that $r\mathbfsl{x}^{\mathbfsl{s}} + g = rx_1 \dots x_d + g\in C$ and the coefficient of  $\mathbfsl{x}^{\mathbfsl{s}}$ in $g$ is $0$. By Lemma \ref{Lempclonoids} we have that $rx_1 \dots x_d \in C$.
		
		Case $\exists j \in [u]$ with $s_j >1$: then we show that there exist $h$ and $g$ such that $h+g \in C$ with $h = r\prod_{i\in[u+1]} x_i^{t_i}$ and $\mathbfsl{t} = (s_1,\dots,s_{j-1},s_j -1,s_{j+1},\dots, s_u, 1)$. Furthermore, $g$ satisfies $\tD(g) \leq d$ and the coefficient of  $\mathbfsl{y}^{(s_1,\dots,s_{j-1},s_j -1,s_{j+1},\dots, s_u, 1)}$ in $g$ is $0$, where we denote by $\mathbfsl{y}$ the vector of variables $(x_1,\dots,x_{u+1})$. Let $g' = f-m \in \ZZ_{p}^n[x_1,\dots,x_l]_p$. Let $g'' := g' \circ_{(j)} (x_j + x_{u+1})$. Thus:
		
		\begin{equation*}
			\begin{split}
				(m+g') \circ_{(j)} (x_j + x_{u+1}) =& r\mathbfsl{x}^{\mathbfsl{s}}\circ_{(j)} (x_j + x_{u+1})+ g' \circ_{(j)} (x_j + x_{u+1})
				\\=&r( \sum_{k \in [s_j]_0} {s_j \choose k}x^{s_j-k}_jx_{u+1}^k)  \cdot \prod_{i \in [u]\backslash\{j\}}x_i^{s_i} + g''
				\\=&r \cdot s_j\cdot \mathbfsl{y}^{(s_1,\dots,s_{j-1},s_j-1,s_{j+1},\dots, s_u, 1)} +
				\\+&r( \sum_{k \in [s_j]_0\backslash\{1\}} {s_j\choose k}x^{s_j -k}_jx_{u+1}^k)\cdot \prod_{i \in [u]\backslash\{j\}}x_i^{s_i} + g''.
			\end{split}
		\end{equation*}
		Note that $s_j$ is invertible in $\ZZ_p$ and that $h+g = s_j^{-1} (m+g') \circ_{(j)} (x_j + x_{u+1})$ is in $C$ with:
		
		\begin{equation*}
			\begin{split}
				&h :=r \cdot \mathbfsl{y}^{(s_1,\dots,s_{j-1},s_j -1,s_{j+1},\dots, s_u, 1)}
				\\&g= s_j^{-1}r( \sum_{k \in [s_j]_0\backslash\{1\}}{s_j\choose k}x^{s_j-k}_jx_{u+1}^k)\cdot \prod_{i \in [u]\backslash\{j\}}x_i^{s_i} + s_j^{-1}g''.
			\end{split}
		\end{equation*}
		Then $h$ satisfies $\tD(h) = d$ with degree $\mathbfsl{t}$. Furthermore, $g$ satisfies $\tD(g) \leq d$ and the coefficient of  $\mathbfsl{y}^{(t_1,\dots,t_{j-1},t_j -1,t_{j+1},\dots, t_w, 1)}$ in $g$ is $0$. Thus $h$ and $g$ are the searched polynomials. This implies that $rx_1 \dots x_d + g''' \in C$ for some $g''' \in \ZZ_{p}^n[X]_p$ with $\tD(g''') \leq d$ and such that the coefficient of  $\mathbfsl{x}^{\mathbf{1}_d}$ in $g'''$ is $0$. By Lemma \ref{Lempclonoids} we have that $rx_1 \dots x_d \in C$ and the claim holds.
	\end{proof}
	
	We are now ready to prove that an $\ZZ_{p}^n$-polynomial linearly closed clonoid generated by an element $f\in\ZZ_{p}^n[X]_p$ contains every monomial of $f$.
	
	\begin{Lem}
		\label{LemMoninZpq}
		Let $p$ be a prime and let $f \in \ZZ_{p}^n[X]_p$ be such that $h = r_{\mathbfsl{m}}\mathbfsl{x}^{\mathbfsl{m}}$ is a monomial of $f$. Then $h \in \langle f \rangle_{ \ZZ_{p}^n}$.
	\end{Lem}
	
	\begin{proof}
		The proof is by induction on the number $k$ of monomials in $f$.
		
		Base step $k = 1$: then clearly the claim holds.
		
		Induction step $k>0$: suppose that the claim holds for every $g$ with $k-1$ monomials. Let $f$ be a polynomial with $k$ monomials. Let $d = \tD(f)$ and let $h$ be a monomial in $f$ with degree $d$ and coefficient $r_h$. By Lemma \ref{LemFoundPcloni}, we have that $r_{h}x_1\cdots x_d \in \langle f \rangle_{\ZZ_{p}^n}$. Clearly, this yields $h \in \langle f \rangle_{\ZZ_{p}^n}$. From the induction hypothesis we have that all $k-1$ monomials of $f -h$ are in $\langle f -h \rangle_{\ZZ_{p}^n} \subseteq \langle f  \rangle_{\ZZ_{p}^n}$. Thus all monomials of $f$ are in $\langle f  \rangle_{\ZZ_{p}^n}$.
	\end{proof}
	
	In this section we have to deal with polynomials whose coefficients are finitary functions from $\FF_i$ to $\ZZ_{p_i}$, where $\FF_{i} = \prod_{j \in [m]\backslash\{i\}}\ZZ_{p_j}$ and $p_1,\dots,p_m$ are distinct primes. In order to connect this strategy with the clones of an expanded group we define a non-standard concept of induced functions of a polynomial. Let $A$ be a set with a fixed element $0$. For every polynomial $f \in \mathbf{R}^{A^n}$ $[x_1,\dots,x_k]_p$ of the form $f =\sum_{\mathbfsl{m} \in [p-1]_0^k}r_{\mathbfsl{m}}\mathbfsl{x}^{\mathbfsl{m}}$ we define its $s$-ary \emph{induced function} $\overline{f}^{[s]}\colon R^s \times A^s \rightarrow R \times A$ by:
	
	\begin{equation*}
		(\mathbfsl{x},\mathbfsl{y}) \mapsto (\sum_{\mathbfsl{m} \in [p-1]_0^k} r_{\mathbfsl{m}}(\mathbfsl{y'})\prod_{i =1}^kx_i^{m_i},\mathbf{0}),
	\end{equation*}
	with $s \geq k,n$ and $\mathbfsl{y'}= (y_1,\dots,y_n)$. We can observe that we induce also the functions $\{r^{\mathbfsl{m}}\}_{\mathbfsl{m} \in [p-1]_0^k}$ coefficients of monomials in $f$ and for this reason we require $s \geq n$. From now on, when not specified, $s =  \max(k,n)$, indeed we want an arity of the induced function sufficiently large to induce both the monomials and the coefficients of the function induced. Next we show a lemma that connects the monomials of an $(\ZZ_p^{\prod_{i=1}^m\ZZ_{q_i}^n})$-polynomial linearly closed clonoid to functions of a clones on $\ZZ_{pq_1\cdots q_m} $.
	
	\begin{Lem}
		\label{LemConnInduced-2}
		Let $p_1,\dots,p_m$ distinct primes, let $\mathbf{R}^{A} = \ZZ_{p_1}^{\prod_{i=2}^m\ZZ_{p_i}^n}$, and let $h,h_1 \in \mathbf{R}^A[X]_{p_1}$ with $h \in \langle h_1 \rangle_{ \mathbf{R}^A}$. Then $\overline{h} \in \Clg(\{\overline{h_1}\})$.
	\end{Lem}
	
	\begin{proof}
		Let $\overline{f }, \overline{g}\in \Clg(\{\overline{h_1}\})^{[s]}$. Then  we can observe that for all $a,b \in \ZZ_{p_1}$ we have that $\overline{af + bg} = h_{((a,b),\mathbfsl{0})} \circ (\overline{f},\overline{g})$, where $\mathbfsl{0} = ((0_{\ZZ_{p_2}}, 0_{\ZZ_{p_2}}), \dots, $ $(0_{\ZZ_{p_m}}, 0_{\ZZ_{p_m}}))$ and $h_{((a,b),\mathbfsl{0})}$ is defined in Remark \ref{RemLinComb}.
		
		Furthermore for all $M \in \ZZ_{p_1}^{s \times l}$ we have that $\overline{f(M \cdot (x_1,\dots,x_l))} = \overline{f} \circ (g_1,\dots,g_s)$ where $g_i:\prod_{i \in [m]}\ZZ_{p_i}^u \rightarrow \prod_{i \in [m]}\ZZ_{p_i} $ such that:
		\begin{equation*}
			g_i: (\mathbfsl{x},\mathbfsl{y}_2,\dots,\mathbfsl{y}_m) \mapsto (M_i(x_1,\dots, x_l)^t,(y_2)_i,\dots,(y_m)_i)
		\end{equation*}
		for all $(\mathbfsl{x},\mathbfsl{y}_2,\dots,\mathbfsl{y}_m )\in \prod_{i \in [m]}\ZZ_{p_{i}}^u$ ,where $M_i$ is the $i$th row of $M$ and $u = \max(l,n)$.
		
		We know that every clone $C$ containing $\Clo( \ZZ_{p_1\cdots p_m},+)$ is closed under composition and, by Remark \ref{RemLinComb}, contains every linear mapping $h_{(\mathbfsl{a}_1,\dots,\mathbfsl{a}_m)}$ with $\mathbfsl{a}_i \in \ZZ_{p_i}^n$. Then it is clear that if a function $h$ can be generated from $h_1$ with item $(1)$ or $(2)$ of Definition \ref{DefPolyClonoid}, then the induced function $\overline{h}_1$ generates $\overline{h}$ in a clone containing $\Clo( \ZZ_{p_1\cdots p_m},+)$, simply composing $\overline{h}_1$ with the linear mappings of Remark \ref{RemLinComb} from the right and from the left.
	\end{proof}
	
	With the next two lemmata we want to prove that in order to characterize clones containing $\Clo(\ZZ_s,+)$, with $s$ squarefree, we have only to consider induced monomials with certain total degrees.
	
	\begin{Lem}
		\label{LemMonVeri-2}
		Let $d \in \NN\backslash\{1\}$. Then for all $k,l \in \NN$, for all $g \in \ZZ_p^{\prod_{i=1}^m\ZZ_{q_i}^l}$, and for all $\mathbfsl{m} \in [p-1]_0^k\backslash\{\mathbf{0}_k\}$ with $\tD(\mathbfsl{x}^{\mathbfsl{m}})  = u$ congruent to $d$ modulo $p-1$ it follows that:
		\begin{equation*}
			\overline{r\mathbfsl{x}^{\mathbfsl{m}}}\in \Clg(\{\overline{rx_1\cdots x_d}\}).
		\end{equation*}
	\end{Lem}
	
	\begin{proof}
		We can observe that composing $\overline{rx_1\cdots x_d}$ with itself we obtain that
		
		\begin{equation*}
			\overline{r^{l+1} x_1\dots x_{d+l(d-1)}} \in \Clg(\overline{rx_1\cdots x_d})
		\end{equation*}
		for all $l \in \NN$. Since $r^p = r$ yields $r^{s(p-1)+1} = r$ for all $s \in \NN$, it follows for $l = s(p - 1)$ that
		
		\begin{equation*}
			\overline{rx_1\cdots x_{d+s(p-1)(d-1)}} \in \Clg(\overline{rx_1\cdots x_d})
		\end{equation*}
		
		Let $s \in \NN$ be such that $d + s(p - 1)(d - 1) \geq \sum_{i=1}^k m_i$. Set the first $m_1$ variables in $\{x_1 , \dots , x_{d+s(p-1)(d-1)}\}$ to $x_1$, the next $m_2 $ variables to $x_2$, and so forth with the last $d + s(p - 1)(d - 1) - \sum_{i \in [k-1]} m_i$ variables set to $x_k$. This yields
		
		\begin{equation*}
			\overline{r\mathbfsl{x}^{\mathbfsl{m}}} \in \Clg(\overline{rx_1\cdots x_d}).
		\end{equation*}
	\end{proof}

	\begin{Lem}
		\label{Lemfcontmon-2}
		Let $p_1,\dots,p_m$ distinct primes, let $n \in \NN$, let $f\colon \prod_{i=1}^m\ZZ_{p_i}^n 
		$ $\rightarrow \prod_{i=1}^m\ZZ_{p_i}$ be an $n$-ary function, and let $g = (p_2\cdots p_m)^{p_1-1}f$. Let $\mathbf{R} = \ZZ_{p_1}$, $A = \prod_{i=2}^m \ZZ_{p_i}^n$, and $h \in \mathbf{R}^A[X]_{p_1}$ such that $\overline{h} = g$. Let $h'$ be a monomial of $h$ with coefficient $r$ and $d = \tD(h')$. Then it follows that:
		
		\begin{equation*}
			\overline{rx_1\cdots x_d}\in \Clg(\{f\}).
		\end{equation*}
		
	\end{Lem}
	
	\begin{proof}
		Let $n$, $h$, and let $f$ be as in the hypothesis. By Lemma \ref{LemFoundPcloni}, we have that $rx_1\cdots x_d \in \langle h' \rangle_{\mathbf{R}^A}$. By Lemma \ref{LemMoninZpq}, $h'\in \langle h \rangle_{\mathbf{R}^A}$ and thus, by Lemma \ref{LemConnInduced-2},  $\overline{rx_1\cdots x_d}\in \Clg(\{f\})$.
	\end{proof}
	
	We are now ready to prove the main result of this section which allows us to provide a bound for the lattice of all clones containing the addition of a squarefree abelian group.
	
	Let $s =p_1\cdots p_m$ be a product of distinct prime numbers. Then for all $i \in [m]$ and $j \in [p_i]_0$ we define $\rho_{(i,j)}\colon \mathcal{L}( \ZZ_{s},+) $ $\rightarrow \mathcal{L}(\ZZ_{p_i},\FF_i)$ by:
	
	\begin{equation}
		\label{defembeGen-2}
		\rho_{(i,j)} (C) := \bigcup_{n\in\NN}\{f \colon\FF_i^n \rightarrow \ZZ_{p_i}\mid \overline{fx_1\cdots x_j} \in C\}
	\end{equation}
	for all $C \in \mathcal{L}( \ZZ_{s},+)$. Let $\rho\colon \mathcal{L}( \ZZ_{s},+)\rightarrow \prod_{i=1}^m\mathcal{L}(\ZZ_{p_i},\FF_i)^{p_i+1}$ be defined by $\rho(C) = (\rho_{(1,0)}(C),\dots,$ $\rho_{(1,p_1)}(C),$ $\dots,\rho_{(m,0)}(C),\dots,\rho_{(m,p_m)}$ $(C))$, for all $C \in \mathcal{L}( \ZZ_{s},+)$. 
	
	\begin{proof}[Proof of Theorem \ref{Thmgeneralembedding-2}]
		
		Let $s =p_1\cdots p_m$ be a product of distinct prime numbers. We prove that for all $i \in [m]$ and for all $j \in [p_i]_0$, the map $\rho_{(i,j)}$ is well-defined and thus $\rho$ is well-defined.
		
		Let $C \in \mathcal{L}(\ZZ_{s},+)$. Then we prove that $\rho_{(1,j)}(C)$, with $0\leq j$, is a $(\ZZ_{p_1},\FF_1)$-linearly closed clonoid. To this end let $n \in \NN$, $f,g \in \rho_{(1,j)}(C)^{[n]}$ and $a, b \in \mathbb{Z}_{p_1}$. Then $\overline{fx_1\cdots x_j}, \overline{gx_1\cdots x_j} \in C$. From the closure with $+$ we have that $\overline{(af+bg)x_1\cdots x_j} \in C$ and thus $af + bg \in \rho_{(1,j)}(C)^{[n]}$ and item $(1)$ of Definition \ref{DefClo-2} holds. Furthermore, let $u,n \in \NN$, $f \in \rho_{(1,j)}(C)^{[u]}$, $A_r \in \mathbb{Z}^{u \times n}_{p_r}$, for all $r \in [m]\backslash\{1\}$, and let $g\colon \prod_{k=2}^m\ZZ_{p_k}^n \rightarrow \ZZ_{p_1}$ be defined by:
		
		\begin{equation*}
			g\colon (\mathbfsl{x}_2,\dots,\mathbfsl{x}_m) \mapsto f(A_2\cdot \mathbfsl{x}_2^t,\cdots,A_m\cdot \mathbfsl{x}_m^t).
		\end{equation*}
		It is clear that $\overline{gx_1\cdots x_j} \in \Clg(\{\overline{fx_1\cdots x_j}\})$ as composition of $\overline{fx_1\cdots x_j}$ and linear mappings of Remark \ref{RemLinComb}. Thus $g \in \rho_{(1,j)}(C)$ which concludes the proof of item $(2)$ of Definition \ref{DefClo-2}. In the same way we can prove that $\rho_{(i,j)}$ is well-defined for all $i \in [m]$ and $j \in [p_i]_0$. Hence $\rho$ is well-defined. 
		
		We prove that $\rho$ is injective. Let $C, D \in \mathcal{L}(\ZZ_{s},+)$ with $\rho(C) = \rho(D)$. Let $f \in C^{[n]}$. By Lemma \ref{Lem2Genexprofb} we have that there exist $m$ sequences of functions $\{f_{(i,{\mathbfsl{h}_i})}\}_{\mathbfsl{h}_i \in [p_i-1]_0^n}$ from $\prod_{j \in [m]\backslash \{i\}}\ZZ_{p_j}^n$ to $\ZZ_{p_i}$, for all $i \in [m]$, such that $f$ satisfies for all $(\mathbfsl{x}_1,\dots,\mathbfsl{x}_m) \in \prod_{i=1}^m\ZZ_{p_{i}}^n$:
		
		\begin{align*}
			f(\mathbfsl{x}_1,\dots,\mathbfsl{x}_m) &= (\sum_{\mathbfsl{h}_1 \in [p_1-1]_0^n} f_{(1,{\mathbfsl{h}_1})}(\mathbfsl{x}_2,\dots,\mathbfsl{x}_m)\mathbfsl{x}_1^{\mathbfsl{h}_1}, \dots,
			\\ &\sum_{\mathbfsl{h}_m \in [p_m-1]_0^n} f_{(m,{\mathbfsl{h}_m})}(\mathbfsl{x}_1,\dots,\mathbfsl{x}_{m-1})\mathbfsl{x}_m^{\mathbfsl{h}_m}).
		\end{align*} 
		Let $w \in \mathbf{R}^A[X]_{p_1}$ be such that $\overline{w}^{[n]}  = \prod_{i=2}^mp_i^{p_1-1}f$, where $\mathbf{R}^A = \ZZ_{p_1}^{\FF_1^n}$. 
		
		Let $h = f_{\mathbfsl{l}}\mathbfsl{x}^{\mathbfsl{l}}$ be a monomial of $w$ and let $s = \tD(h)$. We prove that $\overline{h} \in D$ by case distinction.
		
		Case $s=0,1$: from Lemma \ref{Lemfcontmon-2} it follows that $\overline{h} \in C$. By Definition \ref{defembeGen-2}, $f_{\mathbfsl{l}} \in \rho_{(1,s)}(C) = \rho_{(1,s)}(D)$ and thus $\overline{h} \in D$.
		
		Case $s>1$:  let $d \in \NN$ be such that $2 \leq d \leq p_1$ and $d=s$ modulo $p_1-1$. By Lemma \ref{Lemfcontmon-2}, $C \supseteq \Clg(\{\overline{f_{\mathbfsl{l}}x_1\cdots x_s}\})$. Thus, by Lemma \ref{LemMonVeri-2}, $C \supseteq \Clg(\{\overline{f_{\mathbfsl{l}}x_1\cdots x_d}\})$ and thus $f_{\mathbfsl{l}} \in \rho_{(1,d)}(C) =  \rho_{(1,d)}(D)$. Hence $\overline{f_{\mathbfsl{l}}x_1\cdots x_d} \in D$ and, by Lemma \ref{LemMonVeri-2}, it follows that $\overline{f_{\mathbfsl{l}}\mathbfsl{x}^{\mathbfsl{l}}}  \in D$. This holds for a generic induced monomial in $\prod_{i=2}^mp_i^{p_1-1}f$ and thus the function $\prod_{i=2}^mp_i^{p_1-1}f\in D$. With the same strategy we can prove that $\prod_{i \in [m]\backslash \{j\}}p_i^{p_j-1}f \in D$ for all $j \in [m]$ and thus $f = \sum_{j \in [m]} \prod_{i \in [m]\backslash \{j\}}$ $p_i^{p_j-1}f \in D$. Hence $C \subseteq D$. With the same proof we have the other inclusion and thus $\rho$ is injective.
	\end{proof}
	
	Note that $\rho$ is only an injective function and not a lattice embedding. This happens because the $(\ZZ_{p_i},\FF_i)$-linearly closed clonoids that describe a clone in $\mathcal{L}(\ZZ_s,+)$ have several closure properties that are not preserved by the $(\ZZ_{p_i},\FF_i)$-linearly closed clonoid lattice join.

	\begin{proof}[Proof of Corollary \ref{Corfinale-2}]
		
		The proof follows from Theorems \ref{Thmgeneralembedding-2}, \ref{ThEmbClonoids-2} and we observe the fact that the only clones in common in the embeddings of Theorem \ref{ThEmbClonoids-2} is the clone of all linear mappings. For this reason we subtract $m-1$ from the left hand side of the inequalities.	
	\end{proof}

	\begin{Cor}
		
		\label{Corfinale2}
		Let $s = p_1\cdots p_m \in \NN$ be a product of distinct primes and let $\FF_i = \prod_{j \in [m]\backslash \{i\}} \ZZ_{p_j}$ for all $i \in [m]$. Then the number of clones containing $\Clo(\ZZ_{s},+)$ is bounded by:
		
		\begin{equation*}
			|\mathcal{L}(\ZZ_{s},+)| \leq \prod_{i=1}^m(\sum_{1 \leq r \leq n_i}{{ n_i}\choose{r}}_{p_i})^{p_i+1}
		\end{equation*}
		where $n_i = \prod_{j \in [m]\backslash \{i\}}p_j$ and 
		\begin{equation*}
			{{n}\choose{k}}_q = \prod_{i=1}^k \frac{q^{n-k+i}-1}{q^i-1}.
		\end{equation*}
	\end{Cor}
	
	\begin{proof}
		The proof follows from Corollary \ref{Corfinale-2} and \cite[Theorem $1.4$]{Fio.CSOF2}.
	\end{proof}
	
	We can observe that the bound of Corollary \ref{Corfinale2} is not always reached. We are ready to prove the Dichotomy of Theorem \ref{ThmDichot} and the main result of this paper.
	
	\begin{proof}[Proof of Theorem \ref{ThmDichot}]
		The proof follows from Corollary \ref{Corfinale2} for an abelian group of squarefree order. By \cite{Bul.PCCT} and \cite{Kre.CFSO} a group $\mathbf{G}$ of non-squarefree order has infinitely many expansions up to term equivalence.
	\end{proof}
	
	This nice dichotomy in the behaviour of the expansions of a finite abelian group shows how different the expansions are in case of a squarefree abelian group and in case of a not squarefree one.
	
	With the next two results we can also find a concrete bound for the arity of the generators that we need to characterize these clones.

	\begin{Thm}
		\label{Cor3-2}
		Let $s = p_1\cdots p_m$ be a product of distinct prime numbers and let $\FF_i = \prod_{j\in [m]\backslash \{i\}}\ZZ_{p_j}$. Then a clone $C$ containing $\Clo(\ZZ_{s},+)$ is generated by $S = \bigcup_{i =1}^m S_i$ where:
		
		\begin{equation*}
			S_i :=\bigcup_{j=0}^{p_i}\{\overline{rx_1\cdots x_j} \mid r\colon\FF_i \rightarrow \ZZ_{p_i}, \overline{rx_1\cdots x_j} \in C\}.
		\end{equation*}
		
	\end{Thm}
	
	\begin{proof}
		
		Let $C$ be a clone containing $\Clo(\ZZ_{s},+)$ and let $f \in C^{[n]}$. By Remark \ref{Lem2Genexprofb}, for all $i \in [m]$ there exists  a sequence $\{f_{(i,{\mathbfsl{h}_i})}\}_{\mathbfsl{h}_i \in [p_i-1]_0^n}$ of functions from $\FF_i^n$ to $\ZZ_{p_i}$,  such that $f$ satisfies for all $(\mathbfsl{x}_1,\dots,\mathbfsl{x}_m) \in \prod_{i=1}^m\ZZ_{p_i}^n$:
		
		\begin{equation*}
			\begin{split}
				f(\mathbfsl{x}_1,\dots,\mathbfsl{x}_m) = & (\sum_{\mathbfsl{h}_1 \in [p_1-1]_0^n} f_{(1,{\mathbfsl{h}_1})}(\mathbfsl{x}_2,\dots,\mathbfsl{x}_m)\mathbfsl{x}_1^{\mathbfsl{h}_1}, \dots,\\&\sum_{\mathbfsl{h}_m \in [p_m-1]_0^n} f_{(m,{\mathbfsl{h}_m})}(\mathbfsl{x}_1,\dots,\mathbfsl{x}_{m-1})\mathbfsl{x}_m^{\mathbfsl{h}_m}).
			\end{split}
		\end{equation*} 
		Let $w \in \mathbf{R}^A[X]_{p_1}$ be such that $\overline{w}^{[n]}  = \prod_{i=2}^mp_i^{p_1-1}f$, where $\mathbf{R}^A = \ZZ_{p_1}^{\FF_1^n}$. Let $h = f_{\mathbfsl{l}}x^{\mathbfsl{l}}$ be a monomial of $w$ and let $s = \tD(h)$. Then, by Lemmata \ref{LemMoninZpq} and \ref{LemConnInduced-2}, we have that $\overline{h} \in C$. Furthermore, let $d \in \NN_0$ be such that if $s \not= 0,1$, then $2 \leq d \leq p_1$ and $d=s$ modulo $p_1-1$. If $s=0,1$ then $s = d$. Thus, by Lemmata \ref{LemMonVeri-2} and \ref{Lemfcontmon-2} it follows that $\Clg(\overline{h}) = \Clg(\{\overline{f_{\mathbfsl{l}}x_1\cdots x_s}\}) = \Clg(\{\overline{f_{\mathbfsl{l}}x_1\cdots x_d}\}) $. Then let us consider the $(\ZZ_{p_1},\FF_1)$-linearly closed clonoid generated by $f_{\mathbfsl{l}}$. By \cite[Theorem $1.2$]{Fio.CSOF2}, there exists a set unary functions $F$ from $\FF_1$ to $\ZZ_{p_1}$ such that $\Cid(\{F\}) = \Cid(\{f_{\mathbfsl{l}}\})$. Hence, by the embedding of Theorem \ref{Thmgeneralembedding-2} \eqref{defembeGen-2}, we have that  $\Clg(\{\overline{f_{\mathbfsl{l}}x_1\cdots x_i}\}) =\Clg(\{\overline{gx_1\cdots x_i}\mid g \in F\})$ for all $i \in [p_1]_0$. Hence $\overline{h} \in \Clg(S_1)$ and thus $\prod_{i=2}^mp_i^{p_1-1}f \in \Clg(S_1)$ since $\Clg(S_1)$ contains every induced monomial in $\prod_{i=2}^mp_i^{p_1-1}f$.
		
		In the same way we can observe that $\prod_{j \in [m]\backslash \{i\}}p_j^{p_i-1}f \in \Clg(S_i)$ for all $i \in [m]$ and thus $f = \sum_{i\in [m]}\prod_{j \in [m]\backslash \{i\}}p_j^{p_i-1}f \in \Clg(\bigcup_{i \in [m]} S_i)$ and the claim holds.
	\end{proof}
	
	The proof of Corollary \ref{CorArFun-2} follows directly from Theorem \ref{Cor3-2} and gives an important connection between a clone $C$ containing $\Clo(\ZZ_{s},+)$ and its subsets of generators $S_i$, where $s$ is a product of distinct primes. Theorem \ref{Cor3-2} gives a possibly redundant list of generators for a clone containing $\Clo(\ZZ_{s},+)$ which shows how deep the link between clonoids and clones is. The generators of Theorem \ref{Cor3-2} are a product of a unary member of a set of generators for an $(\ZZ_{p_i},\FF_i)$-linearly closed clonoid and a monomial generating a clone on $\ZZ_{p_i}$. This unifies the characterization in \cite{Kre.CFSO} and \cite[Theorem $1.2$]{Fio.CSOF2} and is the main reason why Theorem \ref{Thmgeneralembedding-2} works. This also justifies the use of polynomials of $\ZZ_p^n[X]$ done in this section to represent functions of a clone of a squarefree abelian group and gives a different perspective to these functions.
	
	Using \cite[Theorem $1.3$]{Fio.CSOF} we can refine Corollary \ref{Corfinale-2} to the following version for clones containing the addition of $\ZZ_{pq}$, with $p,q$ distinct primes.
	
	\begin{Cor}
		\label{Corfinale}
		Let $p$ and $q$ be distinct prime numbers. Let $\prod_{i =1}^n p_i^{k_i}$ and $\prod_{i =1}^s r_i^{d_i}$ be the factorizations of $g_p= x^{q-1} -1$ in $\mathbb{Z}_p[x]$ and of $g_q = x^{p-1} -1$ in $\mathbb{Z}_q[x]$ for irreducible $p_i$, $q_i$, respectively. Then:
		
		\begin{equation*}
			\begin{split}
				&2(\prod_{i =1}^n(k_i +1) + \prod_{i =1}^s(d_i +1) ) - 1\leq |\mathcal{L}(\ZZ_{pq},+)| \leq \\&\leq 2^{p+q+2}\prod_{i =1}^n(k_i +1)^{p+1}\prod_{i =1}^s(d_i +1)^{q+1} \leq 2^{qp+q+p}.
			\end{split}
		\end{equation*}
	\end{Cor}
	
	\section*{Acknowledgements}
	
	The author thanks Erhard Aichinger, who inspired this paper, and Sebastian Kreinecker for many hours of fruitful discussions. The author thanks the referee for his/her useful suggestions.

	\bibliographystyle{alpha}
	
\end{document}